\documentclass[a4paper]{amsart}

\usepackage{amsthm}
\usepackage{amsfonts}
\usepackage{amsmath,amsthm,amssymb}
\usepackage{comment}
\usepackage{fancybox}
\usepackage{epsf}
\usepackage[dvips]{graphicx}

\newtheorem{theorem}{Theorem}[section]
\newtheorem{lemma}[theorem]{Lemma}
\newtheorem{proposition}[theorem]{Proposition}

\theoremstyle{definition}

\theoremstyle{remark}

\numberwithin{equation}{section}

\begin{document}


\title{An elementary proof for that all unoriented spanning surfaces of   
a link are related by attaching/deleting tubes and 
M\"{o}bius bands}

\author{Akira Yasuhara}

\address{Tokyo Gakugei University, Department of Mathematics, Koganei, Tokyo 184-8501, Japan\\
yasuhara@u-gakugei.ac.jp}

\maketitle

\begin{abstract}
Gordon and Litherland showed that all compact, unoriented, possibly non-orientable surfaces 
 in $S^3$ bounded by a link are realted by  attaching/deleting tubes and 
half twisted bands.  In this note we give an elementary proof for this result. 
\end{abstract}



\section{Introduction}
For a link $L$ in $S^3$, a compact, unoriented (possibly non-orientable or disconnected) 
surface in $S^3$ with boundary $L$ is called a {\em spanning surface} of $L$. 
A spanning surface is called a {\em Seifert surface} if it is orientable. 
Two spanning surfaces are {\em $S^*$-equivalent} if they are related by a finite sequence 
of the following two moves and ambient isotopies:\\
(1)~attaching/deleting a tube as illustrated in Figure~\ref{equi}~(1),\\
(2)~attaching/deleting a half twisted band locally as illustrated in Figure~\ref{equi}~(2). \\
In particular, two Seifert surfaces are {\em $S$-equivalent}
if they are related by a finite sequence 
of ambient isotopies and  attaching/deleting a tube that preserves 
orientability of surfaces.

\begin{figure}[th]
\centerline{\includegraphics[trim=0mm 0mm 0mm 0mm, width=.6\linewidth]{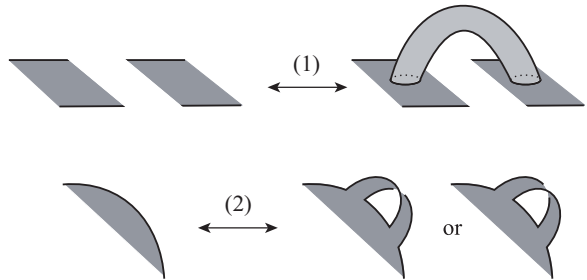}}
  \caption{(1)~attach/delete a tube; (2)~attach/delete a half twisted band locally}\label{equi}
\end{figure}

It is well known that all Seifert surfaces of a link are {\em $S$-equivalent} 
\cite{L}, \cite{R}.
This result is generalized to unoriented spanning surfaces \cite{GL}.

\begin{theorem}(Gordon and Litherland \cite{GL})\label{main}
All spanning surfaces of a link are $S^*$-equivalent.
\end{theorem}

In \cite{BFK}, Bar-Natan, Fulman and Kauffman gave an elementary proof of 
the result that all Seifert surfaces of a link are $S$-equivalent.
Here we will give an elementary proof of Theorem~\ref{main}.
This note is inspired by the article \cite{BFK}. 

\section{Checkerboad surfaces}

For a diagram $D$ in the $2$-sphere $S^2$ of a link $L$, 
we color the regions of $S^2\setminus D$ with black and white like a checkerboard. 
Then we obtain two spanning surfaces of $L$ that correspond to the black regions  
 and white regions of $S^2\setminus D$. We call these two spanning surfaces the 
{\em checkerboard surfaces} for $D$. 
From now on, we will not care the color of these surfaces and 
describe both surfaces by shading.

The local changes, which are naturally induced by Reidemeister moves I, II and III, 
for checkerboard surfaces as 
illustrated in Figure~\ref{R-move}  are called $R$-moves of type I, II and III respectivery.

\begin{figure}[th]
\centerline{\includegraphics[trim=0mm 0mm 0mm 0mm, width=.8\linewidth]{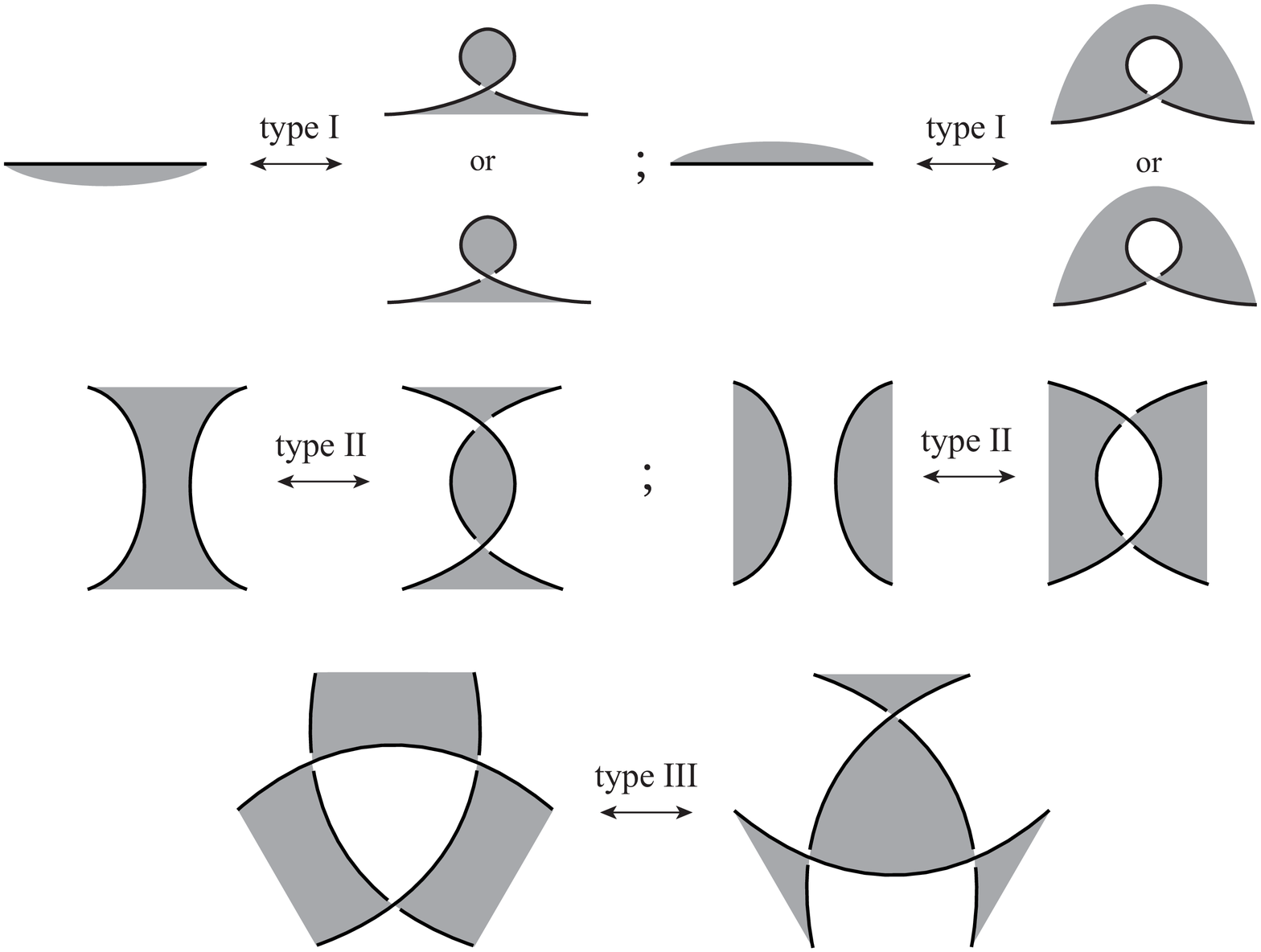}}
  \caption{$R$-moves}\label{R-move}
\end{figure}

\begin{proposition}\label{lemma:BtoW}
All checkerboard surfaces for all diagrams of a link are related by a finite sequences 
of $R$-moves and ambient isotopies. 
\end{proposition}

\begin{proof}
It is obvious that for two link diagrams $D$ and $D'$ of a given link, a checkerboad surface for $D'$ is 
related to one of two checkerboad surfaces for $D$.  
Hence it is enough to show that the two checkerboard surfaces for $D$ are related by a finite sequences 
of $R$-moves and ambient isotopies. 
Figure~\ref{BtoW} shows that the two checkerboard surfaces for the diagram $D$ 
are related by a finite sequence of $R$-moves and ambient isotopy.
\end{proof}

\begin{figure}[th]
\centerline{\includegraphics[trim=0mm 0mm 0mm 0mm, width=.8\linewidth]{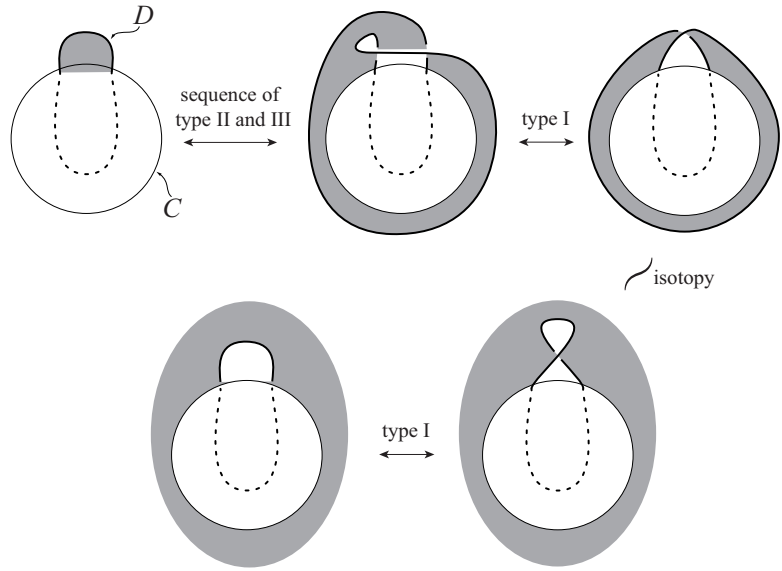}}
  \caption{$C$ is a simple closed curve in $S^2$ that separates $D$ into 
one trivial arc and the other. }\label{BtoW}
\end{figure}


\begin{lemma}\label{lemma:RbyS}
Each $R$-move for checkerboard surfaces can be realized by 
$S^*$-equivalence.
\end{lemma}

\begin{proof}
We start with a easy but important claim mentioned in \cite{BFK}. 
While they consider only orientable surfaces, the claim holds for 
non-orientable surfaces.
 
\medskip
\noindent{\em Claim.} 
The two local changes for spanning surfaces of a link as illustrated in Figure~\ref{observation}
can be realized by attaching/deleting a tube. 
\medskip

\begin{figure}[th]
\centerline{\includegraphics[trim=0mm 0mm 0mm 0mm, width=.9\linewidth]{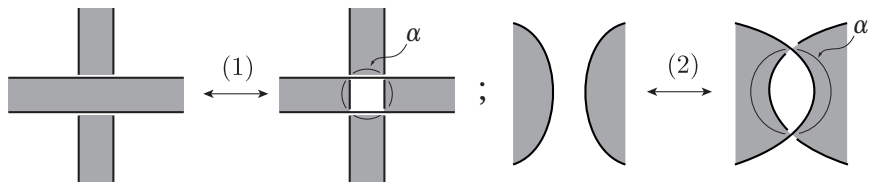}}
  \caption{For each local change, the tube appears as the regular neighborhood
of the loop $\alpha$ in the right-side surface. }\label{observation}
\end{figure}

One of $R$-moves of type I (resp. II) is related by ambient isotopy and the other $R$-move 
of type I (resp. II) is realized by attaching/deleting a half twisted band 
(resp. a tube as the local change of Figure~\ref{observation}~(2)).

For $R$-move of type III, Figure~\ref{deform} completes the proof. 
Here we use attaching/deleting a half twisted band and 
the local change of Figure~\ref{observation}~(1). 
\end{proof}

\begin{figure}[th]
\centerline{\includegraphics[trim=0mm 0mm 0mm 0mm, width=.8\linewidth]{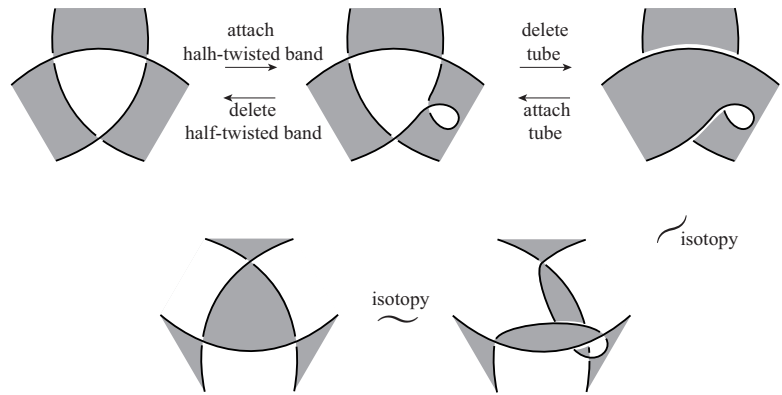}}
  \caption{$R$-move of type III given by $S^*$-equivalence}\label{deform}
\end{figure}

By combining Proposition~\ref{lemma:BtoW} and Lemma~\ref{lemma:RbyS}, we have the 
following proposition.

\begin{proposition}\label{checckerboad}
All checkerboard surfaces for all diagrams of a link are $S^*$-equivalent. 
\end{proposition}

\section{Proof of Theorem~\ref{main}}

The following proposition, together with Proposition~\ref{checckerboad}, 
completes the proof of Theorem~\ref{main}. 

\begin{proposition}
Any spanning surface of a link is 
$S^*$-equivalent to a checkerboard surface for a diagram of the link. 
\end{proposition}

\begin{proof}
A spaning surface $F$ can be assumed as a surface which is 
one disk with several bands attached since $F$ is a compact surface with 
boundary. For example, see Figure~\ref{disk-band}. 
Moreover, we may assume that $F$ has a diagram whose singular parts appear 
in only bands as half twists of bands or crossings between bands as illustrated in  Figure~\ref{band}.
Applying local changes of Figure~\ref{observation}~(1) to all crossings between bands, 
we have a checkerboard surface. 
\end{proof}

\begin{figure}[th]
\centerline{\includegraphics[trim=0mm 0mm 0mm 0mm, width=.4\linewidth]{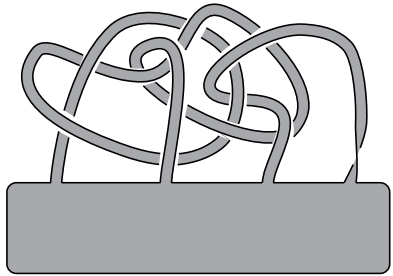}}
  \caption{A disk with two bands}\label{disk-band}
\end{figure}

\begin{figure}[th]
\centerline{\includegraphics[trim=0mm 0mm 0mm 0mm, width=.5\linewidth]{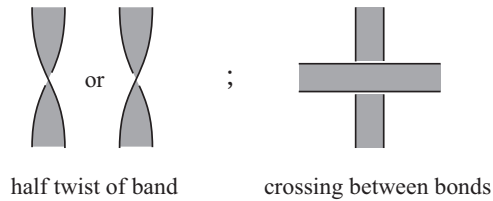}}
  \caption{singular parts}\label{band}
\end{figure}


\section*{Acknowledgments}
The author is partially supported by a Grant-in-Aid for Scientific Research (C) 
($\#$23540074) of the Japan Society for the Promotion of Science.

\end{document}